\documentclass[11pt,a4paper]{amsart}
\usepackage{amssymb}

\newtheorem{theorem}{Theorem}[]
\newtheorem{cor}{Corollary}
\newtheorem{lemma}{Lemma}

\newtheorem{problem}{Problem}

\renewcommand{\geq}{\geqslant}
\renewcommand{\leq}{\leqslant}

\renewcommand{\P}{\mathbb{P}}

\newcommand{\R}{\mathbb{R}}

\newcommand{\E}{\mathbb{E}}

\newcommand{\tr}{\mathrm{tr }}
\newcommand{\cov}{\mathrm{Cov }}

\newcommand{\erf}{\mathrm{Erf}}

\newcommand{\la}{\lambda}
\newcommand{\LL}{\mathcal{L}}
\newcommand{\A}{\mathcal{A}}

\newcommand{\vol}{\mathrm{Vol\,}}
\newcommand{\conv}{\mathrm{conv\,}}

\begin{document}

\title{Extremal cross-polytopes and Gaussian vectors}

\author{Gergely Ambrus}
\email[G. Ambrus]{ambrus@renyi.hu}

\address{Alfr\'ed R\'enyi Institute of Mathematics, Hungarian Academy of Sciences, Re\'altanoda u. 13-15, 1053 Budapest, Hungary}
\thanks{Research was supported by OTKA grants 75016 and
76099.}

\maketitle

\begin{abstract}
For $n \geq 1$, let  $\xi_1, \dots, \xi_n $ be independent, identically
distributed standard normal variables. Among nonnegative real vectors $u =(u_1,
\dots, u_n)$ of norm 1, the quantity $\E  \|(u_1 \xi_1, \ldots, u_n \xi_n \|_\infty$ is maximised when $u$ has at most two non-zero entries, and it is minimised when $u$ is proportional to $(1, \ldots, 1)$. Further generalisations of this result are also discussed. As a corollary, a lower bound on the mean width of a general
convex body $K$ is derived in terms of the successive inner radii of $K$.
\end{abstract}

\section{Context and motivation}

Let $K$ be a convex body in $\R^n$. A natural way to measure how close $K$ is
to a ball is to relate its volume to that of the largest ball inscribed in $K$,
or to the smallest ball circumscribed about $K$. Equivalently, one may as well
relate $\vol(K)$ to the inradius of $K$ and to the circumradius of $K$.
However, the best possible estimates in this case are the trivial ones.

More interesting inequalities can be obtained by taking into account the
successive inner and outer radii of $K$. These are defined as follows. Let
$r(K)$ and $R(K)$ denote the inradius and the circumradius of $K$. Furthermore,
let $\A_i^n$ denote the set of $i$-dimensional affine subspaces of $\R^n$, and
for a subspace $L \in \A_i^n$, denote by $K|L$ the orthogonal projection of $K$
onto $L$. The successive inner and outer radii of $K$ for $1 \leq i \leq n$ are
given by
\[
r_i(K) = \max_{L \in \A_i^n} \, r(K \cap L) \  \textrm{  and  } \ R_i(K) =
\min_{L \in \A_i^n} \, R(K | L)\, .
\]
Note that $r_n(K)=r(K)$, $R_n(K)=R(K)$, $2 r_1(K)$ is the diameter of $K$,
and $2 R_1(K)$ is the minimum width of $K$.

 We also introduce the intrinsic volumes of $K$ for $0 \leq i \leq n$
by
\begin{equation}\label{vi}
V_i(K) = \frac{{n \choose i} \kappa_n}{\kappa_i \kappa_{n-i}} \int_{\LL_i^n}
\vol_i(K|L) \, d\mu_i(L),
\end{equation}
where $\LL_i^n$ is the Grassmannian of all $i$-dimensional linear subspaces of
$\R^n$, equipped with the unique Haar probability measure $\mu_i$, and
$\kappa_n$ is the volume of $B^n$, the unit ball of dimension $n$:
\[
\kappa_n = \frac {\pi ^ {n/2}} {\Gamma(1 + \frac n 2)}\,.
\]
Alternatively, $V_i(K)$ may be expressed as the coefficients in Steiner's
formula
\[
\vol(K + \lambda B^n) = \sum_{i=0}^d \lambda^{n-i} \kappa_{n-i} V_i(K)\,.
\]
The most well-known special cases are: $V_n(K) = \vol (K)$; $2 V_{n-1}(K)$ is
the surface area of $K$; $2 \kappa_{n-1} / (n \kappa_n) V_1(K)$ is the mean
width of $K$; and $V_0(K) = 1$ is the Euler characteristic. For further
references, see Gruber \cite{G07} or Schneider \cite{Sch93}.

Since the intrinsic volumes are the average volumes of projections of $K$ onto
lower dimensional subspaces, it is natural to expect a relationship between
these and the successive radii of $K$.  This link was established by M. Henk
and M. Hern\'andez Cifre \cite{HC08} who proved that the following inequalities
hold:
\begin{align}
 V_i(K) &\leq 2^i s_i (R_1(K), \dots, R_n(K)) \ \textrm{ for every
} 0 \leq i \leq n \, ; \label{ViRi} \\
V_{n-1} (K) &\geq \frac {2^{n-1}}{(n-1)!} \sqrt{s_{n-1}(r_1(K)^2, \dots,
r_n(K)^2)} \, ;\notag \\
V_{n-2} (K) &\geq \frac{2 \sqrt{2}}{\pi} \frac {2^{n-2}}{(n-2)!}
\sqrt{s_{n-2}(r_1(K)^2, \dots, r_n(K)^2)} \, ,\notag
\end{align}
where $s_i$ stands for the $i$th elementary symmetric polynomial:
\[
s_i(\la_1, \dots, \la_n) = \sum_{1 \leq k_1 < \dots < k_i \leq n} \la_{k_1}
\dots \la_{k_i}\, .
\]
The upper bound \eqref{ViRi} is sharp. In order to derive the lower bounds on
$V_i(K)$, the authors of \cite{HC08} apply a sequence of Steiner
symmetrisations to $K$, leading to a convex body $\widetilde{K}$ which contains
the orthogonal cross-polytope spanned by $\pm \, r_1(K)\, e_1, \dots, \pm \,
r_n(K)\,e_n$, where $(e_i)_1^n$ is the standard basis of $\R^n$. Using that
Steiner symmetrisations do not increase the intrinsic volumes (see e.g.
\cite{G07}), one arrives at the following bound:
\begin{equation}\label{ViC}
V_i(K) \geq V_i (C_n(r_1(K), \dots, r_n(K)) )\,,
\end{equation}
where
\[
C_n(\la_1, \dots ,\la_n) = \conv ( \pm \la_i e_i : i = 1, \dots, n )\,.
\]
 Thus, in order to derive the
best bounds provided by this method, one faces the following question:

\begin{problem} \label{pr1}
For $1 \leq i \leq n$, determine the vectors $u=(u_1, \dots, u_n) \in \R^n_+$
minimising
\begin{equation}\label{Vicross}
\frac {V_i (C_n(u_1, \dots, u_n) )}{\sqrt{s_i(u_1^2, \dots, u_n^2)}} \,.
\end{equation}
\end{problem}
We note that the quantity \eqref{Vicross} is invariant under scaling (see e.g.
Corollary 2.1. of \cite{HC08}), thus we may assume that $u \in S^{n-1}$. It is
proved in \cite{HC08} that for $i= n-1$ and $i=n-2$, the minimum of
\eqref{Vicross} is attained when $u$ is a multiple of $(1, \dots, 1)$, that is,
when the cross-polytope is regular. The authors also conjecture that the same
statement should hold for every $i$.

In the present note, we settle the $i=1$ case of Problem~\ref{pr1}, showing
that the minimum is attained in the regular case, whereas the maximum is
attained when the cross-polytope is at most 2 dimensional. Perhaps it is more
convenient to formulate an equivalent question about Gaussian vectors.

\begin{problem}\label{pr2}
Let $n \geq 1$, and  $\xi_1, \dots, \xi_n $ be independent, identically
distributed standard normal variables.  Determine the unit vector $u =(u_1,
\dots, u_n) \in \R^n$ with non-negative coordinates which minimises
\begin{equation}\label{exp}
\E \max_{1 \leq i \leq n} \{ |u_i \xi_i| \}.
\end{equation}
\end{problem}

\noindent
Denoting the random vector $(\xi_1, \dots, \xi_n)$ by $\xi$ and introducing the
 Hadamard product $v \odot w$ of $v, w \in \R^n$ by
\[
(v \odot w)_i = v_i w_i,
\]
the quantity in \eqref{exp} becomes  $\E \| u \odot \xi \|_\infty$. Note that
$u \odot \xi$ is an $n$-dimensional Gaussian random vector with independent
coordinates, whose covariance matrix has trace 1.

The equivalence of Problem~\ref{pr2} and the $i=1$ case of Problem \ref{pr1}
follows by the following standard transformation. For a convex body $K$ in
$\R^n$ and $x \in S^{n-1}$, let $h_K(x)$ denote the support function and
$\rho_K(x)$ the radial function of $K$. If $K$ is symmetric, $\|.\|_K$ denotes
the assigned norm with unit ball $K$. $K^*$ stands for the polar body of $K$.
Furthermore, $\sigma(x) = \sigma_{k-1}(x)$ denotes the $(k-1)$--dimensional
surface (Lebesgue) measure on $S^{k-1}$; note that $\sigma$ is a scaled copy of
the rotationally invariant probablity measure on $S^{k-1}$, with total mass $k
\kappa_k$. With these conventions, using \eqref{vi},
\begin{align*}
\kappa_{n-1}V_1(K) &= \int_{S^{n-1}} h_K(x) d \sigma (x) =\int_{S^{n-1}} \frac
1 {\rho_{K^*}(x)} d \sigma (x) = \int_{S^{n-1}} \|x\|_{K^*}d \sigma (x)\,.
\end{align*}
Since $C_n^*(u_1, \dots, u_n)$ is a rectangular box of half-axes $e_1/u_1,
\dots, e_n/u_n$,
\begin{align}
V_1 (C_n(u_1, \dots, u_n) ) &= \frac 1 {\kappa_{n-1}} \int_{S^{n-1}} \|u \odot x \|_\infty d\sigma (x) \notag\\
&= \frac 1 {(2 \pi)^{(n-1)/2} } \int_{S^{n-1}} \| u \odot x \|_\infty
\int_0^\infty e^{-r^2/2} r^n dr \, d\sigma(x)\notag \\
&= \frac {\sqrt{2 \pi}} {(2 \pi)^{n/2}} \int_{\R^n} e^{- |x|^2 / 2} \, \| u
\odot x
\|_\infty dx \label{v1ex}\\
&= \sqrt{2 \pi} \, \E \| u \odot \xi \|_\infty\,. \notag
\end{align}

Relaxing the independence condition of Problem~\ref{pr2}, we can ask the
following, more general question. We call a multivariate random variable {\em
centred}, if the mean values of its coordinate variables are 0.

\begin{problem}\label{pr3}
Among the $n$-dimensional centred Gaussian random vectors $X$ satisfying $
\tr \;\cov \, X =1$, which ones minimise and maximise $\E \| X \|_\infty$?
\end{problem}

\noindent Problem~\ref{pr2} is a special case of Problem~\ref{pr3} when $\cov X$ is
assumed to be diagonal.

We answer Problems \ref{pr2} and \ref{pr3} formulated above. By {\em standard normal vector} we understand a vector with i.i.d. standard normal coordinate variables.

\begin{theorem}\label{th1} Let $\xi = (\xi_i)_1^n$ be an $n$-dimensional
standard normal vector, and let $u \in \R^n$
be a unit vector. For $n=2$, the expectation $ \E \| u \odot \xi \|_\infty $ is
independent of the choice of $u$. For $n\geq 3$, the expectation is maximised
when at most two coordinates of $u$ are non-zero, and it is minimised when $
u=(\pm 1/\sqrt{n}, \dots, \pm 1/\sqrt{n})$.
\end{theorem}

Thus, the regular cross-polytope is the minimiser for Problem~\ref{pr1}.

\begin{theorem}\label{th2}
Among the $n$-dimensional centred Gaussian random vectors $X$ satisfying $
\tr \;\cov \, X =1$, $\E \| X \|_\infty$ is maximal when $\cov \, X$ is
diagonal with at most two non-zero entries, and minimal when the absolute
values of the coordinates of $X$ are identical almost everywhere.
\end{theorem}

We note that among symmetric convex bodies $K$ in
$\R^n$ in John position (that is, $B_2^n$ is the maximal volume ellipsoid inscribed in $K$),  $\int_{S^{n-1}} \| x\|_K d \sigma(x)$ is minimal for the
cube \cite{Sch95}, see also \cite{B}, pp. 52--53. It also follows that the
cube has minimal mean width among its affine images of the same volume
 \cite{GMR}. This, however, does not imply the above results, as the cube
has the smallest volume among the rectangular boxes to be considered in the
present problem.

 Next, we derive a lower estimate for $V_1(K)$ from Theorem~\ref{th1}.
 Let $\mu$ be the median of $\|x\|_\infty$ on $\R^n$ with respect to the
standard Gaussian measure. For the reader's convenience, we illustrate how to estimate $\mu$, following Ball \cite{B}, pp. 52.  On the one hand, \eqref{v1ex} implies that
\begin{align*}\label{v1est}
V_1\left(C_n\left(\frac 1 {\sqrt{n}} , \dots, \frac 1 {\sqrt{n}}\right) \right)
= \sqrt {\frac {2 \pi}{n}} \frac 1 {(2 \pi)^{n/2}} \int_{\R^n} e^{- |x|^2 / 2}
\, \|  x \|_\infty \, dx > \sqrt {\frac { \pi}{ 2 n}} \, \mu \,.
\end{align*}
On the other hand, $\mu$ satisfies that
\[
\frac 1 2 = \frac 1 {(2 \pi)^{n/2}} \int_{[-\mu,\mu]^n} e^{- |x|^2/2} dx =
\left( \sqrt{\frac 2 \pi} \int _0 ^\mu e^{-s^2 / 2} ds  \right)^n \approx (1 -
e^{-\mu^2/2})^n\,,
\]
thus, from $2^{-1/n} \approx 1 - (\log 2) /n$ we deduce that $\mu \approx
\sqrt{2 \log n}$, and
\begin{equation}\label{enestimate}
V_1\left(C_n\left(\frac 1 {\sqrt{n}} , \dots, \frac 1 {\sqrt{n}}\right) \right)
\approx \sqrt{\pi} \sqrt{\frac {\log n} n} \,.
\end{equation}

Taking  \eqref{ViC} into account, we arrive at the following estimate.
\begin{cor}
There exists an absolute constant $c$, so that for any convex body $K \subset
\R^n$,
\[
V_1(K) \geq c \sqrt{\frac {\log n} n} \sqrt{r_1(K)^2 + \dots + r_n(K)^2}.
\]
\end{cor}
\noindent Numerical calculations show that the value of  $c$ can be chosen to
be $1.74$. This estimate, however, is not optimal, because the approximating orthogonal cross-polytope does not cover $K$ (e.g. the successive inner radii of the regular cross-polytope form a strictly decreasing sequence). In fact, the authors of \cite{HC08} conjecture that the sharp lower bound in terms of the inner radii should be $2\sqrt{r_1(K)^2 + \dots + r_n(K)^2}$; setting $K = C_n(\mu, \mu^2, \ldots, \mu^n)$ for $\mu$ large shows that this bound would be the best possible.

\newpage

\section{Proofs}

We start with a technical lemma.
\begin{lemma}\label{lemm1}
For any $0< q \leq 2$, the function
\[ F(x) = \frac {e^{x^2  / 2}} {x^{q-1}}  \int_0 ^{x
} e^{- t^2/2} dt
\]
is strictly increasing for $x >0$.
\end{lemma}

\begin{proof}
It is easy  to obtain that
\[
F'(x) = \frac 1 {x^{q}}  \left( x + (x^2 - q+1) e^{x^2/2} \int_0^x e^{-t^2/2}
dt \right) =: \frac 1 {x^{q}}\, f(x).
\]
Here $f(0)=0$ and
\[
f'(x) = x^2 + 2 - q+ x (x^2 + 3 - q) \, e^{x^2/2} \int_0^x e^{-t^2/2} dt
\]
which is positive for $x>0$. Thus, $F'(x)>0$ for every $x >0$.
\end{proof}

\begin{proof}[Proof of Theorem~\ref{th1}]
Since $\E \| u \odot \xi \|_\infty$ is a continuous function of $u$ on the unit sphere, it suffices to find the extremum values among the critical points on $S^{n-1}$ (the compactness of the unit sphere implies that the extrema exist). First, we show that the absolute values of the non-zero coordinates of the critical points are all equal. Because of symmetry, we may and do assume that $u_i \geq 0$ for every $i = 1, \dots, n$.

Using that
for a non-negative random variable $X$
\[
\E X= \int _0 ^\infty \P (X >t) dt = \int _0 ^\infty (1 - \P (X \leq t))\, dt\,,
\]
we can express the expectation in question as
\begin{align}\label{epsi}
\E \| u \odot \xi \|_\infty &=  \int_0^\infty \left( 1 - \prod_{i=1}^n \P(|u_i \xi_i|
\leq t)\right) dt \notag
\\
&=  \int_0^\infty \left( 1 - \left(\frac 2 \pi \right)^{n/2} \int_0 ^ {t/u_1}
e^{-s^2/2} ds \dots \int_0^{t/u_n}
 e^{-s^2/2} ds \, \right) dt\\
&= \int_0^\infty \left( 1  - \phi( t / u_1) \dots \phi(t/u_n) \, \right) dt, \notag
\end{align}
where
\[
\phi(a) = \erf\left( \frac a { \sqrt{2}} \right) =  {\sqrt{\frac 2 \pi}}
\int_0^a e^{-s^2/2} ds,
\]
also using the convention that $c/0 = \infty$ for $c \geq0$.

 When $u_1 \neq 0$, integrating by parts leads to
\begin{align}\label{part}
\frac {\partial \, \E \| u \odot \xi \|_\infty}{\partial u_1} &= \sqrt{\frac 2
\pi } \int_0^\infty \frac {t \, e^{ - t^2 / 2 u_1^2}} {u_1^2} \, \phi(t/u_2)
\dots \phi(t/u_n) \, dt \notag \\
&=\left[  - \sqrt{\frac 2 \pi } \, e^{ - t^2 / 2 u_1^2} \phi(t/u_2) \dots
\phi(t/u_n)\right]_0^\infty +\\
& \quad \quad+ \frac 2 \pi  \int_0^\infty e^{ - t^2 / 2 u_1^2} \sum _{i=2}^n
\left(\frac {e^{- t^2/2u_i^2}} {u_i } \prod_{\begin{subarray}{l} j=2\\
j \neq i
\end{subarray}}^n \phi(t/u_j)
\right)\,dt. \notag
\end{align}
Note that on the right hand side, the first summand vanishes.

Next, assume that $u$ is a critical point of $\E \| u \odot \xi \|_\infty$ on $S^{n-1}$ with at least two non-zero coordinates, say,  $u_1 \geq u_2 > 0$. Keeping all the other coordinates fixed and applying the Lagrange
multiplier method to the restricted function, we obtain that
\[
\frac 1 {u_1} \, \frac {\partial \, \E \| u \odot \xi \|_\infty}{  \partial
u_1} = \frac 1 {u_2} \, \frac {\partial \, \E \| u \odot \xi \|_\infty}{
\partial u_2}\, .
\]
Along with \eqref{part}, this implies that
\begin{equation}\label{parteq}
\begin{split}
\int_0^\infty \frac{e^{ - t^2 / 2 u_1^2}}{u_1} &\sum _{i=3}^n
\left(\frac {e^{- t^2/2u_i^2}} {u_i } \prod_{\begin{subarray}{l} j=2\\
j \neq i
\end{subarray}}^n \phi(t/u_j)
\right)\,dt \\
&= \int_0^\infty \frac{e^{ - t^2 / 2 u_2^2}}{u_2} \sum _{i=3}^n
\left(\frac {e^{- t^2/2u_i^2}} {u_i } \, \phi(t/u_1) \prod_{\begin{subarray}{l} j=3\\
j \neq i \end{subarray}}^n \phi(t/u_j) \right)\,dt.
\end{split}
\end{equation}
The quotient of the above integrands is
\[
\frac {u_2}{u_1} \frac{e^{ - t^2 / 2 u_1^2}}{e^{ - t^2 / 2 u_2^2}} \frac{\phi(t
/ u_2)}{\phi(t/u_1)} = \frac s {\phi(s)}\, e^{-s^2/2} \, \frac {e^ {(\mu s)^2 /
2}}{\mu s} \, \phi (\mu s),
\]
where $s= t / u_1$ and $\mu = u_1/u_2$. Setting $q=2$ in Lemma~\ref{lemm1}
implies that for any fixed $s>0$, this is a strictly increasing
function of $\mu$. In particular, $\mu>1$ would imply that the quotient is
strictly greater than 1 for every $s>0$. Thus, equality in \eqref{parteq} can
hold only if $\mu =1$, that is, $u_1 = u_2$.

Therefore, all the non-zero coordinates of the extremal vectors $u$ must be
equal, and in order to find the minimum and the maximum values of $\E \| u
\odot \xi \|_\infty$, it suffices to compute the expectations for the set of
points
\begin{equation*}\label{ueq}
u^k = (\underbrace{1/\sqrt{k}, \dots, 1/\sqrt{k}}_{k}, \underbrace{0, \dots,
0}_{n-k}), \quad k=1, \dots, n.
\end{equation*}
Introducing the notation
\[
E_k = \frac 1 {\sqrt{k}}\, \E \| (\xi_1, \dots, \xi_k) \|_\infty\, ,
\]
our goal is to show that
\begin{equation*}
E_1 = E_2 > E_3 > \dots > E_n.
\end{equation*}
Note that \eqref{v1ex} and \eqref{enestimate} imply that the above inequality
is asymptotically true, as $E_n \approx \sqrt{\log n / 2n}$.

Fix $n \geq 1$, and for $0 \leq \rho \leq 1/\sqrt{n+1}$, introduce
\[
u(\rho) = \left(\sqrt{\frac{1 - \rho^2} n}, \dots, \sqrt{\frac{1 - \rho^2} n},
\rho \right) \in \R^{n+1} \,.
\]
Then, by \eqref{epsi},
\[
\E \| u(\rho) \odot \xi \|_\infty = \int_0^\infty \left( 1 - \phi\left(\frac t
\rho\right) \phi \left(t \sqrt{\frac n {1 - \rho^2}}\,\right)^n \,\right)  dt =:
R(\rho).\]
This leads to
\begin{align*}
R'(\rho) &= \sqrt{\frac 2 \pi } \int_0^\infty \frac t {\rho^2}\, e^{- t^2 / 2
\rho^2}  \phi \left(t \sqrt{\frac n {1 - \rho^2}}\,\right)^n dt \\
& \quad - \sqrt{\frac 2 \pi } \int_0^\infty \frac {t \rho n^{3/2}}{(1 -
\rho^2)^{3/2}} e^{ - n t^2 / 2 (1 - \rho^2)} \phi\left(\frac t \rho\right) \phi
\left(t \sqrt{\frac n {1 - \rho^2}}\,\right)^{n-1} dt\,.
\end{align*}
By partial integration,
\begin{multline*}
\int_0^\infty \frac t {\rho^2}\, e^{- t^2 / 2 \rho^2}  \phi \left(t \sqrt{\frac
n {1 - \rho^2}}\,\right)^n dt  = \left[  - e^{- t^2 / 2 \rho^2} \phi \left(t
\sqrt{\frac n {1 - \rho^2}}\,\right)^n \right]_0^\infty\\
+ \sqrt{\frac 2 \pi} \int_0^\infty  e^{- t^2 / 2 \rho^2} n \sqrt{\frac n {1 -
\rho^2}} e^{ - n t^2 / 2 (1 - \rho^2)} \phi \left(t \sqrt{\frac n {1 -
\rho^2}}\,\right)^{n-1} dt\,,
\end{multline*}
where the first term on the right hand side vanishes. Similarly,
\begin{align*}
\int_0^\infty &\frac {t \, \rho\, n^{3/2}}{(1 - \rho^2)^{3/2}} \, e^{ - n t^2 /
2 (1 - \rho^2)} \phi\left(\frac t \rho\right) \phi \left(t \sqrt{\frac n {1 -
\rho^2}}\,\right)^{n-1} dt \\
&= \left[ -e^{ - n t^2 / 2 (1 - \rho^2)} \rho \sqrt{\frac n {1 -\rho^2}}\,
\phi\left(\frac t \rho\right) \phi \left(t \sqrt{\frac n {1 -
\rho^2}}\,\right)^{n-1}\right]_0^\infty\\
&+ \sqrt{\frac 2 \pi} \int_0^\infty e^{ - n t^2 / 2 (1 - \rho^2)} \sqrt{\frac n
{1 -\rho^2}}\, e^{-t^2/ 2 \rho^2} \phi \left(t \sqrt{\frac n {1 -
\rho^2}}\,\right)^{n-1} dt\\
&+ \sqrt{\frac 2 \pi} \int_0^\infty  e^{ - n t^2 /  (1 - \rho^2)} \rho \frac
{n(n-1)} {1-\rho^2} \, \phi\left(\frac t \rho\right) \phi \left(t \sqrt{\frac n
{1 - \rho^2}}\,\right)^{n-2}dt \,.
\end{align*}
The above equations lead to
\begin{multline}\label{rder}
R'(\rho) = \frac 2 \pi \int_0^\infty e^{ - n t^2 / 2 (1 - \rho^2)}(n-1)
\sqrt{\frac n {1 -\rho^2}} \, \phi \left(t \sqrt{\frac n {1 -
\rho^2}}\,\right)^{n-2}
\\
\cdot \left[  e^{- t^2 / 2 \rho^2} \,  \phi \left(t \sqrt{\frac n {1 -
\rho^2}}\,\right) -  e^{ - n t^2 / 2 (1 - \rho^2)} \rho \sqrt{\frac n {1
-\rho^2}}  \, \phi\left(\frac t \rho\right) \right] dt \,.
\end{multline}
Since $1/\rho \geq \sqrt{n /(1 - \rho^2)}$, Lemma~\ref{lemm1} ($q=2$) implies
that for every $t>0$
\[
\frac 1 {t \sqrt{n / (1 - \rho^2)}} \, e^{  n t^2 / 2 (1 - \rho^2)}  \phi
\left(t \sqrt{\frac n {1 - \rho^2}}\,\right) \leq \frac \rho t \, e^{ t^2 / 2
\rho^2} \,  \phi\left(\frac t \rho\right)\,,
\]
which, by \eqref{rder}, shows that for any $t>0$, the function $R(\rho)$ is
decreasing on the interval $[0, 1/\sqrt{n+1} ]$. Hence, for every $n
\geq 1$,
\[
E_n = R(0) \geq R \left( \frac 1 {\sqrt{n+1}} \right) = E_{n+1}\,.
\]
Furthermore, \eqref{rder} shows that equality holds above if and only if $n=1$.
\end{proof}

\begin{proof}[Proof of Theorem~\ref{th2}]
Let $u_i = \sqrt{(\cov X)_{ii}}$. For the minimum inequality, we may and do
assume that $u_1^2$ is the largest diagonal entry of $\cov X$, hence $u_1 \geq
1/\sqrt{n}$. Since $\|X\|_\infty \geq |X_1|$, and $\E |X_1| = \sqrt{2/\pi}\,
u_1$, we obtain the lower bound $\E \|X\|_\infty \geq \sqrt{2/(n \pi)} $. This
bound is sharp only if for all $i$, $u_i = 1/\sqrt{n}$ and $\|X\|_{\infty} =
|X_i|$ almost everywhere, which yields that for every $i$ and $j$,  $|X_i| =
|X_j|$ almost everywhere.

For the upper bound, let $\xi_1, \dots, \xi_n$ be i.i.d standard normal
variables. By  a theorem of \v{S}id\'ak (\cite{S67}, \cite{S68}), which is a
relative of Slepian's lemma,
\begin{multline*}
\P (\|X\|_\infty \leq t)  = \P (|X_1| \leq t, \dots, |X_n| \leq t) \\\geq
\prod_{i=1}^n \P(|X_i| \leq t) = \prod_{i=1}^n \P(|u_i \xi_i| \leq t).
\end{multline*}
Thus, the question reduces to Problem~\ref{pr2}, and the upper bound provided
by Theorem~\ref{th1} is sharp if the coordinate variables of $X$ are
independent.
\end{proof}

\section{Further remarks}

Theorem~\ref{th1} may also be proved by induction on $n$; the inductive
statement  asserts that for any constant $C\geq0$, the quantity
\[
\int_{\R^n} \max\{C, \|u \odot x\|_\infty\} d \gamma(x),
\]
 where $\gamma$ is the standard $n$-variate Gaussian distribution, is maximal
for $u = (1, 0, \dots, 0)$, and minimal for $ u = (1/\sqrt{n}, \dots,
1/\sqrt{n})$. The initial step is the $n=2$ case. After determining the
possible extremum points using Lagrange multipliers, the remaining statement
amounts to the following.

\begin{lemma}\label{lem2}
For any $c\geq 0$, the following inequality holds:
\[
c + \int_c^\infty ( 1 - (\phi( \sqrt{2} t ))^2  ) \,dt \leq c + \int_c^\infty ( 1 -
\phi (t)) \, dt = \sqrt{\frac 2 \pi}\, e^{ - c^2/2} + c \, \phi(c) \, .
\]
\end{lemma}
\noindent
We choose not to include the somewhat technical proof here; the interested reader can verify the statement by taking second derivatives and analysing the functions.

There are two natural directions to generalise the above results. First,
Problem~\ref{pr1} is open for $2 \leq i \leq n-3$. A method similar to the one
presented here may be applied to these cases as well; however, when computing
the mixed volumes of cross-polytopes, one faces  a formula (see e.g. Corollary
2.1 of \cite{HC08}) which is too complicated to carry out the necessary
analysis. It may be possible to express the mixed volumes in a more suitable
way; in that respect, it is illustrative that \eqref{v1ex} differs from the
$i=1$ case of the above cited formula.

The other direction is  to generalise Problem~\ref{pr2} the following way.

\begin{problem}\label{pr4}
Let $n \geq 1$,  $p,q \in (1, \infty]$, and let $\xi$ be an $n$-variate
Gaussian vector with i.i.d. standard normal coordinate variables. Determine the
maximum and minimum of $\E \| u \odot \xi \|_p $ subject to the condition
$\|u\|_q=1$.
\end{problem}

Clearly, in the $p = q = 1$ case the expectation is independent of the choice
of $u$, whereas Theorem~\ref{th1} provides the answer to the $q=2$, $p =
\infty$ case. We now extend this for $0<q<2$ as well.
\begin{theorem}
For $p = \infty$ and $0< q \leq 2$, the answer to Problem~\ref{pr4} is given
as follows: $\E \| u \odot \xi \|_\infty $ is maximised by $u = (1, 0, \dots,
0)$, and it is minimised by $u = (n^{-1/q}, \dots, n^{-1/q})$.
\end{theorem}

\begin{proof}
The argument applied in the course of the proof of Theorem~\ref{th1}, together
with the general case of Lemma~\ref{lemm1} imply that the extremal vectors are
among the
 $u_q^k$ given by
\begin{equation*}
u_q^k = (\underbrace{k^{-1/q}, \dots, k^{-1/q}}_{k}, \underbrace{0, \dots,
0}_{n-k}), \quad k=1, \dots, n.
\end{equation*}
Introducing
\[
E_{k,q} = k^{-1/ q}\, \E \| (\xi_1, \dots, \xi_k) \|_\infty\, ,
\]
the chain of inequalities
\[
E_{1,q}> E_{2,q}> \dots > E_{n,q}
\]
easily follows by
\[
\left( \frac{k+1} k\right)^{1/q} > \left( \frac{k+1} k\right)^{1/2} \geq
\frac{\E \| (\xi_1, \dots, \xi_{k+1}) \|_\infty}{\E \| (\xi_1, \dots, \xi_k)
\|_\infty} \,. \qedhere
\]
\end{proof}

 It would be natural to expect that  for every $p$ and $q$, the behaviour of
the extremal vectors is similar to the above case.  However, this is very far
from being true. We illustrate this phenomenon by the $n=2$ and $p=2$ case,
where by elementary but tedious calculations one can show the following. Let
$q_L = 3/2$, $q_M = \log 2/(\log \pi - \log 2) \approx 1.53$ and $q_U = 2$. For
$1 \leq q \leq q_L$, the expectation is maximal when $ u = (1,0)$, and there is
exactly one local (and global) minimum at $(2^{-1/q}, 2^{-1/q})$. For $q_L < q
< q_U$, there is a local maximum at both of these directions, and there are two
further local (and global) minimum points. When $q =q_M$, these two maxima are
equal; for $q < q_M$, the global maximum is at $(1,0)$, whereas for $q>q_M$,
the global maximum is at $(2^{-1/q}, 2^{-1/q})$. For $q \geq q_U$, the vector
$(1,0)$ becomes a global minimum point, and the only local extremum points are
$(1,0)$, $(0,1)$ and $(2^{-1/q}, 2^{-1/q})$. For general $p$, a similar pattern
holds with $q_U=2$, and $q_L \rightarrow 2$ as $p \rightarrow \infty$.

The same situation occurs in higher dimensions, with $q_U$ depending on~$n$.
Thus, in particular, it is not true in general that the extremal vectors $u$
are necessarily of the form $(\alpha, \dots, \alpha, 0, \dots, 0)$, up to
permutation and sign changes of the coordinates.

For $p= \infty$ and $q>2$, the distribution of the minimum and maximum points
is unclear. Theorem~\ref{th1} and the discussion above show that for $q \leq
2$, the maximum is achieved at $(1,0,\dots,0)$, whereas the minimum is taken
when $u$ is parallel to $(1, \dots, 1)$. However, for $q=\infty$, the role of
these two directions is clearly swapped. Thus, there must be a transition phase
as $q \rightarrow \infty$, and it is plausible to expect that the behaviour of
the extremal points depends heavily on the dimension as well. That for any
$q>2$, whether all the extremal points have equal absolute values of the non-zero coordinates remains an open question.

\section{Acknowledgements}
I would like to thank F. Fodor and V. V\'igh for communicating the problem to me
and for the fruitful discussions and ideas; K. Ball, M. Henk, A. Litvak, and G.
Pisier for the useful advice; and the anonymous referee for the valuable suggestions. I am grateful for the hospitality of the Mathematical Sciences Research Institute, Berkeley, CA.

\end{document}